\documentclass[10pt]{amsart}

\usepackage{amsmath,amssymb,amsfonts,epsfig, amsthm}
\usepackage{enumerate}
\usepackage{tikz}
\usepackage{graphicx}

\theoremstyle{plain}
\newtheorem{thm}{Theorem}[section]
\newtheorem{lemma}[thm]{Lemma}
\newtheorem{prop}[thm]{Proposition}

\theoremstyle{definition}
\newtheorem{defn}[thm]{Definition}
\newtheorem{exa}[thm]{Example}

\newcommand{\pol}{P} 
\newcommand{\dis}{s(\Sigma)}

\begin{document}

\title[Quasi-trees]{A quasi-tree expansion of the Krushkal polynomial}
\author{CLARK BUTLER}

\begin{abstract}
We introduce a generalization of the Krushkal polynomial to nonorientable surfaces, and prove that this polynomial has a natural quasi-tree expansion. This generalized Krushkal polynomial contains the Bollob\'as-Riordan polynomial of a (possibly nonorientable) ribbon graph as a specialization. The quasi-tree expansion proven here then extends the recent quasi-tree expansions of the Bollob\'as-Riordan polynomial deduced in the oriented case by A. Champanerkar et al. and in the more general unoriented case by E. Dewey and F. Vignes-Tourneret. The generalized Krushkal polynomial also contains the Las Vergnas polynomial of a cellulation of a surface as a specialization; we use this fact to deduce a quasi-tree expansion for the Las Vergnas polynomial. 
\end{abstract}
 
\maketitle 

\numberwithin{equation}{section}
\section{Introduction}
The classical Tutte polynomial was first introduced as a sum over the spanning trees of a graph \cite{Tu}. B. Bollob\'as and O. Riordan extended the Tutte polynomial to ribbon graphs in \cite{BR}, and proved a spanning tree expansion for this extended polynomial. However, recent work on the Bollob\'as-Riordan polynomial has suggested that one-face spanning subgraphs (which are known as \emph{quasi-trees}) are more natural to consider than spanning trees in the case of ribbon graphs. A. Champanerkar, I. Kofman, and N. Stoltzfus gave an expansion of the Bollob\'as-Riordan polynomial over the quasi-trees of an orientable ribbon graph \cite{CKS1}. This expansion was extended to non-orientable ribbon graphs independently by F. Vignes-Tourneret \cite{VT2} and E. Dewey \cite{De}. In \cite{VT2} this quasi-tree expansion was applied to derive an expansion of the Kaufmann Bracket polynomial of a virtual link diagram over connected states. Quasi-trees have seen use outside of graph polynomials: in \cite{CKS2} the quasi-trees of a ribbon graph associated to a link diagram are used to construct the Khovanov homology. 

V. Krushkal introduced a new polynomial invariant of graphs embedded into orientable surfaces in \cite{Kru}. This polynomial satisfies a natural duality relation in the case of a cellularly embedded graph, as well as some mild contraction-deletion properties, and contains the Bollob\'as-Riordan polynomial as a specialization. It was shown in \cite{ACEMS} that the Krushkal polynomial also contains the Las Vergnas polynomial of a cellularly embedded graph as a specialization, and that the specializations giving the Bollob\'as-Riordan polynomial and Las Vergnas polynomial are independent. The Krushkal polynomial has also been extended to triangulations of even-dimensional oriented manifolds, for which a similar duality property holds \cite{KrRe}.

This paper consists of four major sections. In the first, a generalization of the Krushkal polynomial to graphs embedded into any compact (possibly non-orientable) surface is given. We show that this generalized Krushkal polynomial retains all of the familiar properties of the Krushkal polynomial, including the duality relation. This generalization is motivated by a desire to capture the Bollob\'as-Riordan polynomial for all ribbon graphs (not just orientable graphs) as a specialization of a Krushkal-like polynomial.

In the second section, we review the essential properties of the quasi-tree grouping of the family of spanning subgraphs of a ribbon graph with a total ordering on the edges, following \cite{VT2}. The third section contains the main result of the paper. A quasi-tree expansion for the generalized Krushkal polynomial is formulated and proven. The natural duality of the generalized Krushkal polynomial forces this expansion to take a particularly nice form. In the fourth and final section, we apply the main result to obtain a quasi-tree expansion for the Krushkal polynomial and to rederive the Bollob\'as-Riordan quasi-tree expansions given by \cite{CKS1}, \cite{De}, and \cite{VT2}, using the fact that the generalized Krushkal polynomial contains the Bollob\'as-Riordan polynomial as a specialization. Lastly, we extend the main theorem of \cite{ACEMS} to the generalized Krushkal polynomial, thus proving that the generalized Krushkal polynomial contains the Las Vergnas polynomial as a specialization as well, and use this fact to derive a quasi-tree expansion for the Las Vergnas polynomial. 
  
I would like to thank Sergei Chmutov for many invaluable discussions and suggestions which improved this paper greatly.  
\section{The Generalized Krushkal Polynomial}

\subsection{The Krushkal Polynomial}
Let $i:G \longrightarrow \Sigma$ be an embedding of a graph $G$ into a compact orientable surface $\Sigma$. For a spanning subgraph $F$ of $G$, we will define a collection of topological invariants given by the induced embedding of $F$ into $\Sigma$. In general, for any space $X$, we will let $c(X)$ denote the number of connected components of $X$. The restriction of $i$ to $F$ induces a homomorphism $i_{*}$ of the corresponding first homology groups with coefficients in $\mathbb{R}$. Define
\begin{equation}
k(F) = \text{dim}( \text{ker}(i_{*}: H_{1}(F;\mathbb{R}) \longrightarrow H_{1}(\Sigma; \mathbb{R})))
\end{equation}
For a compact orientable surface $M$, let $g(M)$ denote the genus of $M$. Let $\mathcal{F}$ be a regular neighborhood of the image $i(F)$ of the subgraph $F$ in $\Sigma$. The boundary of $\mathcal{F}$ is a disjoint union of circles. We may obtain a surface $\mathcal{S}(F)$ by gluing disks to each of the boundary components of $\mathcal{F}$. Let $\mathcal{S}^{\perp}(F)$ be the surface obtained by gluing in disks to the boundary components of the complement of $\mathcal{F}$ in $\Sigma$. Define
\begin{equation}\label{s}
s(F) = 2g(\mathcal{S}(F)) 
\end{equation}
\begin{equation}\label{sp}
s^{\perp}(F) = 2g(\mathcal{S}^{\perp}(F))
\end{equation}
The Krushkal polynomial is defined as
\begin{defn}\label{krupol}(\cite{Kru})
\[
P_{G,\Sigma}(X,Y,A,B) = \sum_{F \subseteq G} X^{c(F)-c(G)}Y^{k(F)}A^{s(F)/2}B^{s^{\perp}(F)/2}
\]
where the sum ranges over all spanning subgraphs $F$ of $G$. 
\end{defn}

In \cite{ACEMS} it was shown that 
\begin{equation}\label{connker}
k(F) = c(\Sigma \backslash F) - c(\Sigma)
\end{equation}
This identity may also be deduced inductively by the observation that the deletion of a homologically trivial loop separates a surface into two components.

If $\Sigma - i(G)$ is a disjoint union of disks, we say that $i:G \longrightarrow \Sigma$ is a \emph{cellulation} of $\Sigma$. We generally will not refer to the embedding $i$ explicitly and will simply identify $G$ with its image in $\Sigma$. Thus we will speak of $G$ as being a cellulation of $\Sigma$. For a cellulation $G$ of $\Sigma$, we may construct a new cellulation $G^{*}$ of $\Sigma$ as follows: place one vertex in each connected component of $\Sigma - G$, and for each edge $e$ of $G$, connect the vertices of $G^{*}$ corresponding to the components of $\Sigma - G$ on either side of $e$ by an edge $e^{*}$ which cuts $e$ transversely. $G^{*}$ is called the \emph{Poincar\'e dual} of $G$. $P_{G,\Sigma}$ behaves well with respect to Poincar\'e duality: 
\begin{thm}\label{krudual}(\cite{Kru})
\[
P_{G,\Sigma}(X,Y,A,B) = P_{G^{*},\Sigma}(Y,X,B,A)
\]
\end{thm}

Throughout this paper we will be concerned principally with cellulations of surfaces. In this case the language of embedded graphs and ribbon graphs is interchangeable. A \emph{ribbon graph} is a pair $(G,S)$ in which $G$ is a graph, $S$ is a surface with boundary, and the inclusion $G \longrightarrow S$ is a homotopy equivalence. We will think of ribbon graphs as having a handle decomposition in which the 0-handles correspond to neighborhoods of the vertices and the 1-handles correspond to neighborhoods of the edges. This definition is taken from \cite{Kru}, though we could have started with this handle decomposition of a surface with boundary as a definition of a ribbon graph; see \cite{BR}, \cite{Ch} for more details on this and other equivalent definitions of ribbon graphs. 

Given a cellulation $G$ of a surface $\Sigma$, we obtain a ribbon graph by taking a regular neighborhood of $G$ in $\Sigma$. Conversely, given a ribbon graph $(G,S)$, we obtain a closed surface $\Sigma$ by gluing discs to the boundary components of $S$. The induced embedding of $G$ in the resulting surface is then a cellulation. We will thus refer to $G$ interchangeably as a ribbon graph and as an embedded graph when $G$ is a cellulation. Poincar\'e duality of ribbon graphs takes the form of the following construction: For a ribbon graph $G$, we glue a disc to each of the boundary components of $G$ (which we take to be the vertices of $G^{*}$), and delete the vertices of $G$.

\subsection{The Generalized Krushkal Polynomial}
We seek now a natural extension of the Krushkal polynomial to graphs embedded into nonorientable surfaces.  When $\Sigma$ is orientable and $F$ is a spanning subgraph of a graph $G$ embedded in $\Sigma$, we note that the invariants $s$, $s^{\perp}$ are given by the formulas 
\begin{equation}
s(F) = 2c(F) - \chi(\mathcal{S}(F))
\end{equation}
\begin{equation}
s^{\perp}(F) = 2c(\Sigma \backslash F) - \chi(\mathcal{S}^{\perp}(F))
\end{equation}

where $\chi$ is the Euler characteristic of the surface. These expressions are defined even when $\Sigma$ is not assumed to be orientable. We will thus take these expressions to \emph{define} the invariants $s$ and $s^{\perp}$ in the general case when $\Sigma$ may not be orientable, so that the surfaces $\mathcal{S}(F)$ and $\mathcal{S}^{\perp}(F)$ may not be orientable. When either $\mathcal{S}(F)$ or $\mathcal{S}^{\perp}(F)$ are nonorientable, the corresponding invariant $s(F)$ or $s^{\perp}(F)$ gives the nonorientable genus of the corresponding surface.




Now we extend the Krushkal polynomial to nonorientable surfaces by the formula, 
\begin{equation}
\pol_{G,\Sigma}(X,Y,A,B) = \sum_{F \subseteq G} X^{c(F)-c(G)}Y^{c(\Sigma \backslash F) - c(\Sigma)}A^{s(F)/2}B^{s^{\perp}(F)/2}
\end{equation}
which is well-defined regardless of whether or not $\Sigma$ is orientable. This agrees with Definition \ref{krupol} when $\Sigma$ is orientable, by equation \eqref{connker}. When the surfaces $\mathcal{S}(F)$ and $\mathcal{S}^{\perp}(F)$ are orientable, the corresponding invariants $s(F)$ and $s^{\perp}(F)$ are always even, and $\pol_{G,\Sigma}$ is a polynomial in the variables $X$, $Y$, $A$, and $B$. However, if $\mathcal{S}(F)$ is nonorientable, then $\chi(\mathcal{S}(F))$ may be odd, and consequently $s(F)$ may be odd (and likewise for $s^{\perp}$). Thus, in general $\pol_{G,\Sigma}$ is a polynomial in $X$, $Y$, $A^{1/2}$, and $B^{1/2}$.   
 
The rest of this section is devoted to proving that $\pol_{G,\Sigma}$ retains all of the properties of the Krushkal polynomial for orientable surfaces, when $\Sigma$ is allowed to be nonorientable. The proofs are almost entirely analagous to those in \cite{Kru}, except they have been reformulated in combinatorial language. In fact, our extension of the Krushkal polynomial entirely avoids the language of homology. It can be recast in homological language analagous to Krushkal's, except with $\mathbb{Z}/2\mathbb{Z}$ coefficients instead of coefficients in $\mathbb{R}$. Some remarks on this are made at the end of this section. 

The following proposition states some of the basic properties of $P_{G, \Sigma}$. The proofs of these properties are essentially identical to those in \cite{Kru}, except that we have replaced homological triviality with the equivalent notion (which Krushkal utilizes) that deletion of the loop separates the surface into two components. The proofs are thus omitted.

\begin{prop}
\begin{enumerate}
\item
If $e$ is an edge of $G$ which is neither a loop nor a bridge, then $\pol_{G} = \pol_{G \backslash e} + \pol_{G/e}$
\item
If $e$ is a bridge in $G$, then $\pol_{G} = (1+X)\pol_{G/e}$
\item
If $e$ is a loop in $G$ such that $\Sigma - e$ has two components, then $\pol_{G} = (1+Y)\pol_{G \backslash e}$
\item
If $\Sigma = \Sigma_{1} \sqcup \Sigma_{2}$, $G = G_{1} \sqcup G_{2}$, $G_{1} \subseteq \Sigma_{1}$, $G_{2} \subseteq \Sigma_{2}$, and $\Sigma_{1}$,  $\Sigma_{2}$ are disjoint surfaces, then $\pol_{G,\Sigma} = \pol_{G_{1},\Sigma_{1}} \pol_{G_{2}, \Sigma_{2}}$
\end{enumerate}
 \end{prop} 
 



 
 For a graph $F$, let $e(F)$ denote the number of edges in $F$, and let $v(F)$ denote the number of vertices in $F$. The \emph{nullity} of $F$ is defined by $n(F):= e(F) - v(F) + c(F)$. For a graph $G$, the Tutte polynomial $T_{G}$ of $G$ is given by  
 \begin{equation}
 T_{G}(X,Y) = \sum_{F \subseteq G} X^{c(F) - c(G)}Y^{n(F)}
 \end{equation}
 
 This is a particular normalization of the Tutte polynomial, chosen both so that its relation with $\pol_{G}$ will be as simple as possible and to make the formulas in Section \ref{expan} as simple in appearance as possible. 
 
 We now think of $G$ as being embedded in a compact surface $\Sigma$. We fix some notation and concepts which will be used throughout the rest of this paper. For a spanning subgraph $F$ of $G$, let $bc(F)$ denote the number of boundary components of the regular neighborhood $\mathcal{F}$ of $F$. Since $F$ is a cellulation of $\mathcal{S}(F)$, and the components of $\mathcal{S}(F) - F$ correspond to boundary components of $\mathcal{F}$, we deduce that 
 \begin{equation}
 \chi(\mathcal{S}(F)) = v(F) - e(F) + bc(F)
 \end{equation}
 If $G$ is a cellulation of $\Sigma$, we will denote by $F^{*}$ the spanning subgraph of the dual $G^{*}$ whose edges consist of all edges in $G^{*}$ which do not intersect an edge in $F$. Then $F^{*}$ is a cellulation of $\mathcal{S}^{\perp}(F)$. Thus we deduce 
  \begin{equation}
 \chi(\mathcal{S}^{\perp}(F)) = v(F^{*}) - e(F^{*}) + bc(F^{*})
 \end{equation}
 
 Define $ \dis = 2c(\Sigma) - \chi(\Sigma)$. 
 
 We will now show that the Tutte polynomial of the graph $G$ can be recovered from $\pol_{G,\Sigma}$.
 
  \begin{thm}\label{Tutte}
 \[
 T_{G}(X,Y) = Y^{\dis/2}\pol_{G,\Sigma}(X,Y,Y,Y^{-1})
 \]
 \end{thm}
 \begin{proof}
 We show that the equality holds on monomials. We have 
 \begin{equation}
 Y^{\dis/2}\pol_{G,\Sigma}(X,Y,Y,Y^{-1}) = \sum_{F \subseteq G}X^{c(F)-c(G)}Y^{c(\Sigma \backslash F)-c(\Sigma) + \dis/2+s(F)/2-s^{\perp}(F)/2}
 \end{equation}
 
 Thus we must show that 
 \begin{equation}\label{comb}
 n(F) = c(\Sigma \backslash F)-c(\Sigma) + \dis/2+s(F)/2-s^{\perp}(F)/2
 \end{equation}
 $G$ may not be a cellulation of $\Sigma$. However, we can extend $G$ to be a cellulation of $\Sigma$ by adding edges to $G$. Let $\tilde{G}$ be this extended graph. Then $\tilde{G}$ is a cellulation of $\Sigma$. $\tilde{G}$ has the same vertices as $G$, and thus any spanning subgraph of $G$ is also a spanning subgraph of $\tilde{G}$. $F$ is therefore a spanning subgraph of $\tilde{G}$. $\tilde{G}$ has a dual $(\tilde{G})^{*}$; we will let $F^{*}$ be the corresponding subgraph of the dual (as per our previous remarks). Then 
  \begin{align*}
c(\Sigma \backslash F)-c(\Sigma) + \dis/2+s(F)/2-s^{\perp}(F)/2  & = c(\Sigma \backslash F)-c(\Sigma)\\
&+ c(\Sigma) -\frac{1}{2}(v(\tilde{G})-e(\tilde{G})+bc(\tilde{G}))\\
&+ c(F) - \frac{1}{2}(v(F)-e(F)+bc(F))\\
&- c(\Sigma \backslash F) + \frac{1}{2}(v(F^{*})-e(F^{*})+bc(F^{*})) 
\end{align*}
 We note $v(\tilde{G}) = v(F) = v(G)$ and $e(F^{*}) = e(\tilde{G}) - e(F)$. We have $v(F^{*}) = bc(\tilde{G})$, and $bc(F) = bc(F^{*})$, since a regular neighborhood of $F^{*}$ is given by the complement of a regular neighborhood of $F$. Combining these assertions, it is easy to see that the expression on the right, after cancellations, reduces to $n(F)$. 
\end{proof}

The duality property of $\pol_{G, \Sigma}$ for orientable surfaces $\Sigma$ extends naturally to the general case. Let $G$ be a cellulation of a compact surface $\Sigma$, and $G^{*}$ the Poincar\'e dual graph to $G$ in $\Sigma$. Then we have 
 
 \begin{thm}\label{dual}
 \[
 \pol_{G,\Sigma}(X,Y,A,B) = \pol_{G^{*},\Sigma}(Y,X,B,A)
 \]
 \end{thm}
 
 \begin{proof}
 We claim that the monomial of $\pol_{G, \Sigma}$ corresponding to $F$ is equal to the monomial of $\pol_{G^{*}, \Sigma}$ corresponding to $F^{*}$. It is clear that $\mathcal{S}(F) \cong \mathcal{S}^{\perp}(F^{*})$,  $\mathcal{S}^{\perp}(F) \cong \mathcal{S}(F^{*})$, and $c(\Sigma \backslash F) = c(F^{*})$, since $F$ and $F^{*}$ have complementary regular neighborhoods. Hence $c(F) = c(\Sigma \backslash F^{*})$, $c(\Sigma \backslash F) = c(F^{*})$, $s(F) = s^{\perp}(F^{*})$, and $s^{\perp}(F) = s(F^{*})$. Lastly $c(\Sigma) = c(G^{*})$, since $G$ and $G^{*}$ are both cellulations of $\Sigma$. 
 \end{proof}
 
 When $G$ is a cellulation, we can think of $G$ as a ribbon graph and spanning subgraphs of $G$ as spanning subgraphs of this ribbon graph. In light of Theorem \ref{dual}, we can think of the generalized Krushkal polynomial as a ribbon graph invariant given by   
 \begin{equation}\label{ribbon}
 \pol_{G}(X,Y,A,B) = \sum_{F \subseteq G} X^{c(F)-c(G)} Y^{c(F^{*}) - c(G^{*})} A^{s(F)}B^{s(F^{*})}
 \end{equation}
 where we now think of $s$ as a ribbon graph invariant given by $s(F) = 2c(F) - v(F) + e(F) - bc(F)$. In the remainder of the paper, we will be concerned exclusively with ribbon graphs and evaluations of the Krushkal polynomial on ribbon graphs, for which we will use the form of equation \eqref{ribbon}.
 
 We conclude this section with a few remarks on the topological significance of $P_{G,\Sigma}$ in relation to homology with coefficients in $\mathbb{Z} / 2\mathbb{Z}$. As usual, let $\Sigma$ be a compact surface, let $G$ be a graph embedded in $\Sigma$ by a map $i$, and let $F$ be a spanning subgraph of $G$. The reader may check, by following the techniques in \cite{Kru} with coefficients in $\mathbb{Z} / 2\mathbb{Z}$ instead, that
 \begin{equation}
 c(\Sigma \backslash F) - c(\Sigma) = \text{dim} \; \text{ker}(i_{*}: H_{1}(F;\mathbb{Z} / 2\mathbb{Z}) \longrightarrow H_{1}(\Sigma;\mathbb{Z} / 2\mathbb{Z}))
 \end{equation}
 and that this equality holds regardless of the orientability of $\Sigma$. It is straightforward to compute as well that 
 \begin{align}
 &s(F) = \text{dim} \; H_{1}(\mathcal{S}(F); \mathbb{Z} / 2\mathbb{Z}) \\
 &s^{\perp}(F) = \text{dim} \; H_{1}(\mathcal{S}^{\perp}(F); \mathbb{Z} / 2\mathbb{Z})
 \end{align}
 
 

 
 \section{Quasi-trees}

 We transition now to exclusively using the language of ribbon graphs. The exposition in this section mostly follows that of \cite{VT2}, using the partial duality of S. Chmutov \cite{Ch} to define activities of edges with respect to a given quasi-tree. A binary tree is then constructed  

\subsection{Partial Duality}

 Partial duality was introduced under the name \emph{generalized duality} by S. Chmutov in \cite{Ch} and given the more appropriate name of \emph{partial duality} in \cite{Mo2}. Partial duality generalizes the notion of duality for ribbon graphs. There are several equivalent procedures for obtaining the partial dual of a ribbon graph with respect to a subset of edges; the procedure we follow is that of \cite{BBC}. For alternatives which use a more combinatorial representation of ribbon graphs, see \cite{Ch} and \cite{EMM}.
 
 For a ribbon graph $G$, let $E(G)$ be the set of edges of $G$, and let $H \subseteq E(G)$. We define the \emph{partial dual} of $G$ with respect to $H$ to be the ribbon graph $G^{H}$ obtained from the following procedure: Consider the spanning subgraph $F_{H}$ of $G$ consisting of all the vertices of $G$ and the edges of $H$. The inclusion $F_{H} \longrightarrow G$ embeds the boundary curves of $F_{H}$ into $G$. Glue a disc to each of these closed curves in $G$, and remove the interiors of the vertices of $G$. The resulting surface with boundary $G^{H}$ is a ribbon graph whose vertices are the disks which were glued in, and whose edges either run parallel to the corresponding edges of $G$ (for those edges not in $H$) or cut transversely the corresponding edges of $G$ (for those edges in $H$). 
 
 We list some properties of partial duality.
 
 {\bf Properties} \cite{Ch}
  \begin{enumerate}
  \item
  $G^{\emptyset} = G$
  \item
  $G^{E(G)} = G^{*}$ 
  \item
  $(G^{H})^{H'} = G^{(H \cup H') \backslash (H \cap H)}$
  \item
  Partial duality preserves orientability.
  \item
  $c(G^H) = c(G)$
  \item
  $v(G^H) = bc(F_{H})$
  \item
  $bc(G^H) = bc(F_{E(G) \backslash H})$
  \end{enumerate}
  
  We also record the following property, which we will need later. 
  
  \begin{prop}\label{parconn}
  Let $A,B \subseteq E(G)$ with $A \cap B = \emptyset$. Then $c(G \backslash B) = c(G^{A} \backslash B)$.
  \end{prop}
  \begin{proof} 
  This is immediate from our formulation of partial duality. Since $A$ and $B$ are disjoint, removing the edges of $B$ does not affect the spanning subgraph with edge set $A$. Thus we may obtain the ribbon graph $G^{A} \backslash B$ by first deleting the edges in $B$, then gluing in disks along the boundary components of the spanning subgraph corresponding to $A$, and then removing the interiors of the vertices of $G$. Since the last two of these operations preserve connected components, we conclude that $c(G \backslash B) = c(G^{A} \backslash B)$.
  \end{proof}
 
 \begin{defn}
 A quasi-tree $Q$ is a ribbon graph with $bc(Q) = 1$. For $G$ a ribbon graph, denote the set of spanning subgraphs of $G$ which are quasi-trees by $\mathcal{Q}_{G}$.
 \end{defn}
 It's clear that any spanning tree of a ribbon graph must be a quasi-tree, but not every quasi-tree is a spanning tree. 
 If $Q$ is a quasi-tree of $G$, then $G^{E(Q)}$ is a one-vertex ribbon graph. For two edges $e$, $e'$ in $E(G)$ we say that $e$ \emph{links} $e'$ if, traversing the boundary of the vertex of $G^{E(Q)}$, the boundaries of the edge ribbons $e$ and $e'$ are met alternately. It's trivial to see that $e$ links $e'$ if and only if $e'$ links $e$. 

\begin{defn}
Let $G$ be a ribbon graph and $Q \in \mathcal{Q}_{G}$ be one of its quasi-trees. Let $\prec$ be a total order on the edges $E(G)$ of $G$. $e \in E(G)$ is \emph{live} if $e$ does not link any lower-ordered edge; otherwise it is \emph{dead}. $e$ is \emph{internal} if $e \in E(Q)$ and \emph{external} otherwise.
\end{defn}

	This notion of activities is due to F. Vignes-Tourneret \cite{VT2} and is equivalent to those given in \cite{CKS1} (in the orientable case) and \cite{De}. We will further distinguish edges by orientability: for a quasi-tree $Q$ we say that an edge is orientable (resp. nonorientable) if the edge forms an orientable (resp. nonorientable) loop in $G^{E(Q)}$.
	\begin{exa}\label{graph}We end this section by computing the activities of an example. Consider the ribbon graph below, with the given ordering on the edges,
	
	\raisebox{0pt}[80pt][30pt]{\begin{picture}(80,15)(-60,20)
        \put(0,0){\includegraphics[scale = .3]{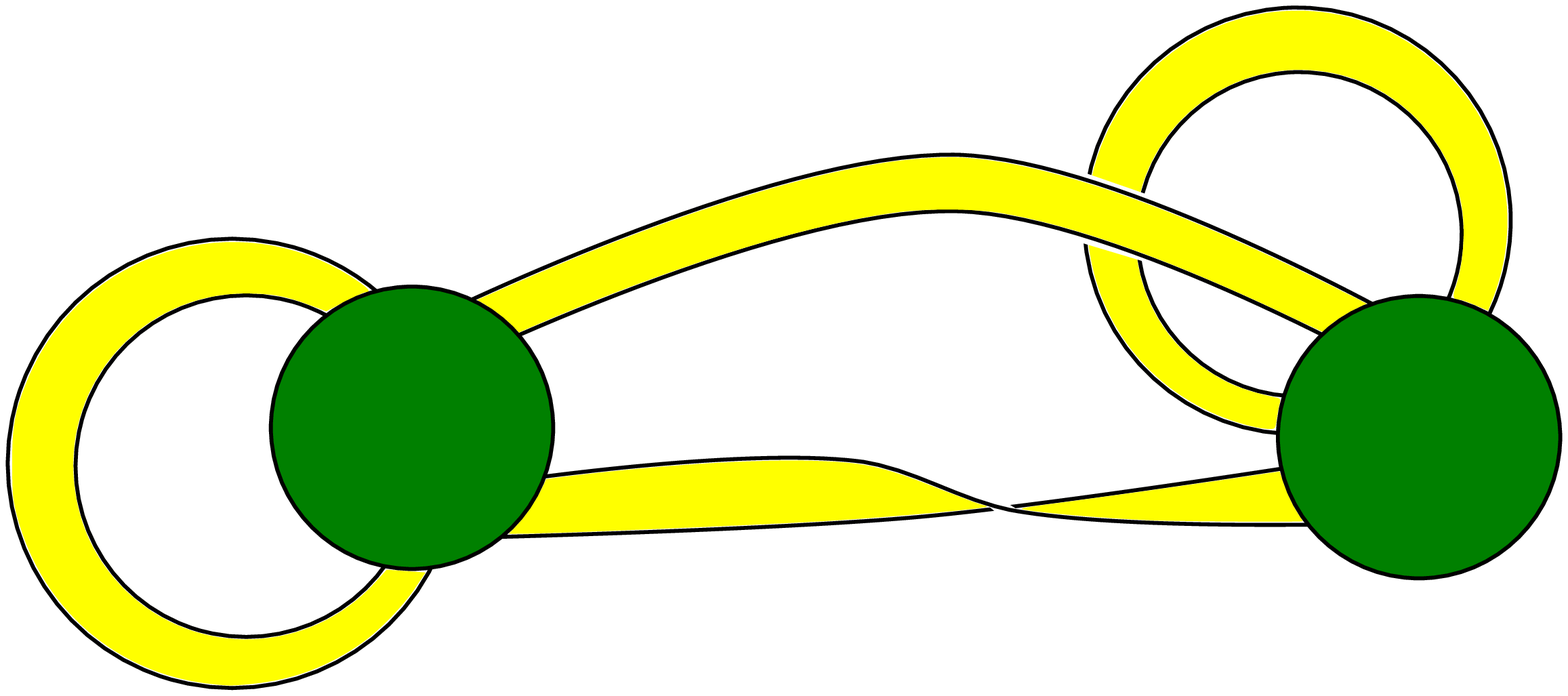}} 
        \put(-10,30){$1$}
        \put(110,70){$2$}
        \put(110,5){$3$}
        \put(190,80){$4$}
     \end{picture}}
     
     This ribbon graph has a quasi-tree $\left\{2,3,4\right\}$. Taking the partial dual with respect to these edges gives the following ribbon graph,
     
	\raisebox{0pt}[80pt][70pt]{\begin{picture}(80,15)(-80,50)
        \put(0,0){\includegraphics[scale = .3]{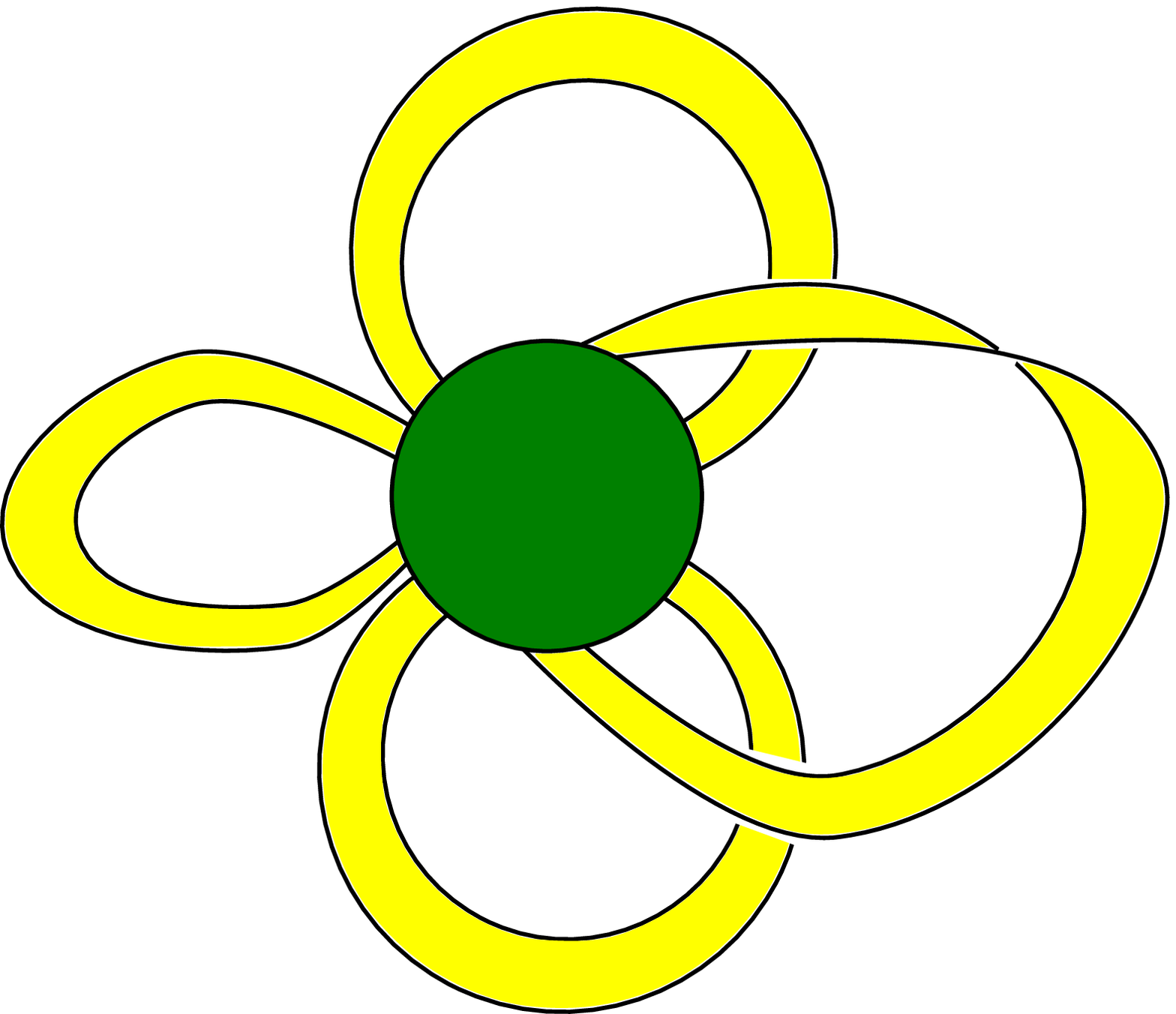}} 
        \put(-10,60){$1$}
        \put(60,120){$2$}
        \put(60,-10){$3$}
        \put(140,60){$4$}
     \end{picture}}
     
     We see that $1$ is externally live, $2,3$ are internally live, and $4$ is dead. Note that $4$ is also a non-orientable edge with respect to this quasi-tree. 
     
     \end{exa}
	\subsection{Binary Tree of Partial Resolutions}
	\begin{defn}
	A \emph{partial resolution} of $G$ is a function $\rho : E(G) \longrightarrow \left\{0,1,*\right\}$. A \emph{resolution} is a partial resolution which only takes values in $\left\{0,1\right\}$.
	\end{defn}    
 There is a natural bijection between resolutions of $G$ and spanning subgraphs of $G$. Namely, to a resolution $\gamma$ we associate the spanning subgraph of $G$ whose edges are those $e \in E(G)$ for which $\gamma(e) = 1$. For a partial resolution $\rho$ and a resolution $\gamma$, we say that $\rho$ \emph{contains} $\gamma$ if $\rho^{-1}(i) \subseteq \gamma^{-1}(i)$, for $i = 0,1$. If we think of edges $e$ for which $\rho(e) = *$ as "`unresolved edges"', then $\rho$ contains $\gamma$ if and only if the unresolved edges of $\rho$ can be resolved to achieve the resolution $\gamma$. Resolutions which correspond to quasi-trees will be called \emph{quasi-tree resolutions}.
 
 Let $\rho$ be a partial resolution, and let $e$ be an edge for which $\rho(e) = *$. Let $\rho_{0}^{e}$ and $\rho_{1}^{e}$ be the partial resolutions in which $e$ is resolved to have value 0 and 1, respectively (all other edges remain unchanged). We say that $e$ is \emph{nugatory} (with respect to $\rho$) if at least one of the partial resolutions $\rho_{0}^{e}$ and $\rho_{1}^{e}$ do not contain a quasi-tree resolution.
 
 For a total order $\prec$ on $E(G)$ we construct a binary tree $\mathcal{T}(G)$ whose nodes are partial resolutions of $G$ by the following inductive procedure: The root of $\mathcal{T}(G)$ is the partial resolution assigning the value * to all edges. Beginning with the highest order edge, we resolve each edge $e$ to be either $0$ or $1$, except when $e$ is nugatory, in which case we leave $e$ unresolved and move to the next highest order edge. We end on a leaf when all subsequent edges are nugatory, then choose the lowest order non-nugatory edge on the tree which has yet to be resolved, and repeat this process, concluding when all unresolved edges are nugatory. The result is a binary tree for which each leaf $\rho$ corresponds to a quasi-tree $Q_{\rho}$ of $G$, as all unresolved edges of $\rho$ are nugatory and hence may be resolved uniquely to a quasi-tree. 
 
 The following theorem describes the structure of the leaf $\rho$ in terms of the activities of $Q_{\rho}$. It forms the mathematical backbone for the quasi-tree expansions of ribbon graph polynomial invariants. In the orientable case it was proven in \cite{CKS1}; the extension to the non-orientable case was done concurrently in \cite{VT2} and \cite{De}.
 
 \begin{thm}\label{live}
 Let $\rho$ be a leaf of $\mathcal{T}(G)$. Then $\rho(e) = *$ if and only if $e$ is live and orientable in $Q_{\rho}$.
 \end{thm}
 
 Informally, unresolved edges in a leaf are live and orientable, while resolved edges are either dead or live and nonorientable. 

 At this point it is prudent to fix some notation. Given a total order $\prec$ on $E(G)$ and a quasi-tree $Q \in \mathcal{Q}_{G}$, we consider
 \begin{enumerate} 
 \item
 $\mathcal{DI}(Q)$, the set of internally dead edges.
 \item
 $\mathcal{I}_{o}(Q)$, the set of internally live orientable edges.
 \item
  $\mathcal{I}_{n}(Q)$, the set of internally live nonorientable edges.
  \item $\mathcal{DE}(Q)$, the set of externally dead edges.
 \item $\mathcal{E}_{o}(Q)$, the set of externally live orientable edges. 
 \item $\mathcal{E}_{n}(Q)$, the set of externally live nonorientable edges.
 
 \end{enumerate}  
  Note that each of these categories is mutually exclusive; a quasi-tree gives rise to a six-fold partition of the set $E(G)$ determined by $\prec$. We can rephrase Theorem \ref{live} in the following way, by observing that each resolution of $G$ belongs to a unique leaf of $\mathcal{T}(G)$.
 
 \begin{thm}\label{activ}
 Let G be a connected ribbon graph and $\mathcal{S}_{G}$ its set of spanning subgraphs. Given a total order on E(G), $\mathcal{S}_{G}$ is in one-to-one correspondence with $\sqcup_{Q \in \mathcal{Q}_{G}} \mathcal{I}_{o}(Q) \times \mathcal{E}_{o}(Q)$. Namely to each spanning subribbon graph $F$ there corresponds a unique quasi-tree $Q_{F}$, for which there exists $S \subseteq \mathcal{I}_{o}(Q_{F}) \cup \mathcal{E}_{o}(Q_{F})$ such that $E(F) = \mathcal{DI}(Q_{F}) \cup \mathcal{I}_{n}(Q_{F}) \cup S$. 
 \end{thm}
 
 This theorem implies that a total order $\prec$ determines a unique grouping of the spanning subgraphs of $G$ according to the quasi-trees of $G$. 
 
 We end this section by collecting two lemmas from \cite{VT2} that we will need. For a subset $H \subseteq E(G)$, let $F_{H}$ be the spanning subgraph of $G$ whose edge set is $H$. Let $S = S_{1} \sqcup S_{2}$, where $S_{1} \subseteq \mathcal{I}_{o}(Q)$ and $S_{2} \subseteq \mathcal{E}_{o}(Q)$. Lastly, we adopt the following notation: $\mathcal{VI}(Q) := \mathcal{DI}(Q) \cup \mathcal{I}_{n}(Q)$ and $\mathcal{VE}(Q):= \mathcal{DE}(Q) \cup \mathcal{E}_{n}(Q)$. These two lemmas are grounded in the following observation: if $e$ and $e'$ are both live and orientable with respect to a quasi-tree $Q$, then they cannot link each other. Thus in the ribbon graph $G^{E(Q)}$, $e$ and $e'$ are orientable loops which do not link one another. 
 
 \begin{lemma}\label{conn}
 \[
 c(F_{\mathcal{VI}(Q) \cup S}) =  c(F_{\mathcal{VI}(Q) \cup S_{1}})
 \]
 \end{lemma}
 
 \begin{proof}
 We think of these spanning subgraphs as being obtained from $G$ by deleting edges. Thus $F_{\mathcal{VI}(Q) \cup S} = G \backslash (\mathcal{VI}(Q) \cup S)^{c}$. By Proposition \ref{parconn}, 
 \begin{equation}
 c(G \backslash (\mathcal{VI}(Q) \cup S)^{c}) = c(G^{\mathcal{VI}(Q) \cup S_{1}} \backslash (\mathcal{VI}(Q) \cup S)^{c})
 \end{equation}
 $G^{\mathcal{VI}(Q) \cup S_{1}}$ is obtained from $G^{E(Q)}$ by taking the partial dual with respect to a subset of the internally live, orientable edges. Since no live edges link each other in $G^{E(Q)}$, every externally live orientable edge will remain a loop in $G^{\mathcal{VI}(Q) \cup S_{1}}$. Thus deleting edges in $\mathcal{E}_{o}(Q)$ from this graph does not affect the number of connected components, regardless of whether any other edges have been removed. Thus,
 \begin{align*}
 &c(G^{\mathcal{VI}(Q) \cup S_{1}} \backslash (\mathcal{VI}(Q) \cup S)^{c}) = c(G^{\mathcal{VI}(Q) \cup S_{1}} \backslash ((\mathcal{VI}(Q) \cup S)^{c} \cup S_{2})) \\
  &= c(G^{\mathcal{VI}(Q) \cup S_{1}} \backslash (\mathcal{VI}(Q) \cup S_{1})^{c}) = c(G \backslash (\mathcal{VI}(Q) \cup S_{1})^{c})
 \end{align*}
 \end{proof}

 \begin{lemma}\label{bc}
 \[
 bc(F_{\mathcal{VI}(Q) \cup S}) =  bc(F_{\mathcal{VI}(Q)}) - \left| S_{1} \right| + \left| S_{2} \right|
 \]
 \end{lemma}
 
 \begin{proof}
 We have 
 \begin{equation}
 bc(F_{\mathcal{VI}(Q) \cup S}) = v(G^{\mathcal{VI}(Q) \cup S}) = v((G^{E(Q)})^{\mathcal{I}_{o}(Q)\backslash S_{1} \cup S_{2}}) = bc(H)
 \end{equation}
  where $H$ is the spanning subgraph of $G^{E(Q)}$ whose edge set is $\mathcal{I}_{o}(Q)\backslash (S_{1} \cup S_{2}$). Since all of these edges are live and orientable, $H$ is a single vertex graph whose edges are orientable loops, none of which are linked. Thus 
  \begin{equation}
  bc(F_{\mathcal{VI}(Q) \cup S}) = bc(H) = \left| \mathcal{I}_{o}(Q) \right| - \left|S_{1}\right| + \left|S_{2}\right|  
  \end{equation}
 Taking $S = \emptyset$ gives $bc(F_{\mathcal{VI}(Q)}) = \left| \mathcal{I}_{o}(Q) \right|$, from which the lemma follows. 
 \end{proof}
 
 \section{Quasi-tree Expansions}\label{expan}
 
 We will deduce a quasi-tree expansion for the polynomial $\pol_{G}$ when $G$ is a ribbon graph (or equivalently, a cellularly embedded graph). We will assume $G$ is connected, so that $c(G) = 1$ and quasi-tree subgraphs actually exist. Observe that since $c(G) = c(G^{*})$, we will also then have $c(G^{*}) = 1$. Since for spanning subgraphs $F$ of $G$, $bc(F) = bc(F^{*})$, we have that $Q$ is a quasi-tree of $G$ if and only if the corresponding subgraph $Q*$ of $G^{*}$ is also a quasi-tree. Thus there is a natural bijective correspondence between quasi-trees of $G$ and $G^{*}$. 
 
 Given a total order $\prec$ on $E(G)$, there is a natural total order on $E(G^{*})$ induced by the bijection $e \longleftrightarrow e^{*}$, where $e \in E(G)$ and $e^{*}$ is the unique edge of $G^{*}$ that intersects $e$ on $\Sigma$. Let $\prec^{*}$ be this total order on $E(G^{*})$. As the next lemma shows, the activities of $Q^{*}$ with respect to $\prec^{*}$ are essentially the same as those for $Q$ with respect to $\prec$. In the lemma, equality should be understood under the correspondence $e \longleftrightarrow e^{*}$.
 
 \begin{lemma}\label{dualactiv}
 
 \begin{enumerate}
 \item
 $\mathcal{DI}(Q^{*}) = \mathcal{DE}(Q)$
 \item
 $\mathcal{I}_{o}(Q^{*}) = \mathcal{E}_{o}(Q)$
 \item
  $\mathcal{I}_{n}(Q^{*}) = \mathcal{E}_{n}(Q)$
 \item 
  $\mathcal{DE}(Q^{*}) = \mathcal{DI}(Q)$
  \item
 $\mathcal{E}_{o}(Q^{*}) = \mathcal{I}_{o}(Q)$
 \item 
 $\mathcal{E}_{n}(Q^{*}) = \mathcal{I}_{n}(Q)$
 \end{enumerate}  
 
 \end{lemma}
 
 \begin{proof}
 To compute the activities of $Q^{*}$, we must investigate $(G^{*})^{E(Q^{*})}$. Passing to the edges of $G$ via the edge correspondence preserves the total order on edges and thus the activities. Then, 
 \[
 (G^{*})^{E(Q^{*})} = (G^{E(G)})^{E(G)-E(Q)} = G^{E(Q)}
 \]
 The equalities of the lemma all follow immediately from this equality and the observation that since $e^{*} \in Q^{*} \iff e \notin Q$, the notions of internal and external are flipped in $G^{*}$
 \end{proof}
 
 \begin{exa} Consider the ribbon graph from example \ref{graph}. Its dual is  
 
 \raisebox{0pt}[80pt][70pt]{\begin{picture}(80,15)(-80,50)
        \put(0,0){\includegraphics[scale = .3]{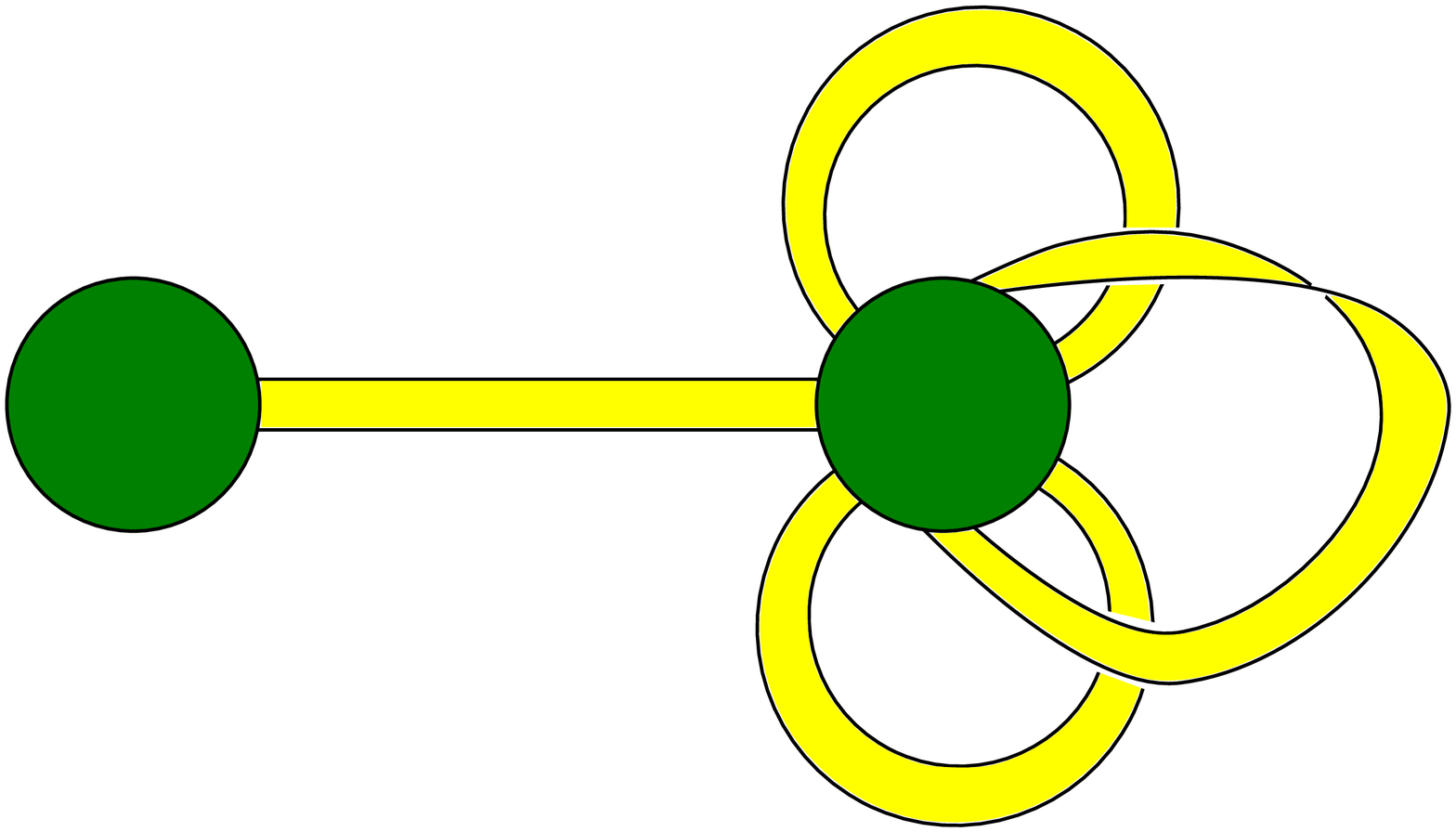}} 
        \put(60,70){$1$}
        \put(150,120){$2$}
        \put(150,-10){$3$}
        \put(210,60){$4$}
     \end{picture}}
 
 The dual quasi-tree to the quasi-tree considered in example \ref{graph} is the quasi-tree $\left\{1\right\}$. With respect to this quasi-tree, $1$ is internally live, $2,3$ are externally live, and $4$ is dead.
 \end{exa}
 We adopt the notational convention that for a subset $H \subseteq E(G)$, the corresponding spanning subgraph of $G^{*}$ will be denoted $R_{H}$. Note $(F_{H})^{*} = R_{E(G) \backslash H}$ under our notation. For a ribbon graph $G$ and a quasi-tree $Q$ of $G$, let $G_{Q}$ be the (ordinary) graph whose vertices are the connected components of $F_{\mathcal{VI}(Q)}$ and whose edges are those edges in $\mathcal{I}_{o}(Q)$. We will also consider $G^{*}_{Q^{*}}$, the graph whose vertices are the connected components of $R_{\mathcal{VE}(Q)}$ and whose edges are those edges in $\mathcal{E}_{o}(Q)$, considered as edges in $G^{*}$.
 \begin{exa}For $G$ and $Q$ the ribbon graph and quasi-tree considered in example \ref{graph}, the associated graphs are shown below  
 
\raisebox{-27pt}[0pt][20pt]{\begin{picture}(80,15)(-200,0)
        \put(0,0){\includegraphics[width=80pt]{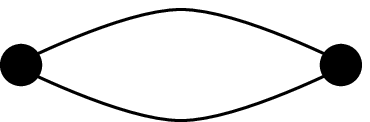}} 
        \put(37,30){$2$}
        \put(37,-10){$3$}
        \put(37,-28){$G_{Q}$}
     \end{picture}}

 \raisebox{0pt}[0pt][40pt]{\begin{picture}(80,15)(-40,0)
        \put(0,0){\includegraphics[width=80pt]{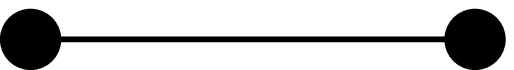}} 
        \put(37,-8){$1$}
        \put(36,-35){$G^{*}_{Q^{*}}$}
     \end{picture}}
 
 \end{exa}
 
 We now state the main theorem of this paper. 
 
 \begin{thm}\label{quasi}
 \[
 \pol_{G}(X,Y,A,B) = \sum_{Q \in \mathcal{Q}_{G}} T_{G_{Q}}(X,A) T_{G^{*}_{Q^{*}}}(Y,B) A^{s(F_{\mathcal{VI}(Q)})/2} B^{s(R_{\mathcal{VE}(Q)})/2} 
 \]
 \end{thm}
 
 \begin{proof}
 By equation \eqref{ribbon} we can write $\pol_{G}$ as 
 
 \begin{equation}
 \pol_{G}(X,Y,A,B) = \sum_{F \subseteq G} X^{c(F)-1} Y^{c(F^{*}) - 1} A^{s(F)/2}B^{s(F^{*})/2}
 \end{equation}
 Using Lemma \ref{activ} and the equalities of Lemma \ref{dualactiv}, we can rewrite this sum as 
 \begin{equation}\label{dblesum}
 \sum_{Q \in \mathcal{Q}_{G}} \sum_{S \subseteq \mathcal{I}_{o}(Q) \cup \mathcal{E}_{o}(Q)} X^{c(F_{\mathcal{VI}(Q) \cup S})-1} Y^{c(R_{\mathcal{VE}(Q) \cup S^{c}})-1} 
 A^{s(F_{\mathcal{VI}(Q) \cup S})/2} B^{s(R_{\mathcal{VE}(Q) \cup S^{c}})/2}
 \end{equation}
 
 where $S^{c} = \mathcal{I}_{o}(Q) \cup \mathcal{E}_{o}(Q) - S$.
 
 Write $S = S_{1} \sqcup S_{2}$, where $S_{1} \subseteq \mathcal{I}_{o}(Q)$ and $S_{2} \subseteq \mathcal{E}_{o}(Q)$. Likewise $S^{c} = S_{1}^{c} \sqcup S_{2}^{c}$. Application of Lemma \ref{conn} yields
 \begin{equation}\label{conn2}
 c(F_{\mathcal{VI}(Q) \cup S}) =  c(F_{\mathcal{VI}(Q) \cup S_{1}})
 \end{equation}
 and also yields, upon applying the lemma to the subgraphs of $G^{*}$,
 \begin{equation}\label{dualconn}
 c(R_{\mathcal{VE}(Q) \cup S^{c}}) =  c(R_{\mathcal{VE}(Q) \cup S_{2}^{c}})
 \end{equation}
 
 Since internal edges are external edges of the dual and vice versa, it is easy to see that the equation \eqref{dualconn} simply rephrases Lemma \ref{conn} for the dual graph. 
 
 Let $W$  be the spanning subgraph of $G_{Q}$ whose edge set is $S_{1}$, and let $W^{*}$ be the spanning subgraph of $G^{*}_{Q^{*}}$ whose edge set is $S_{2}^{c}$. Recall that
 \begin{equation}
 s(F_{\mathcal{VI}(Q) \cup S}) = 2c(F_{\mathcal{VI}(Q) \cup S_{1}}) - v(G) + e(F_{\mathcal{VI}(Q) \cup S}) - bc(F_{\mathcal{VI}(Q) \cup S})
 \end{equation}
 
 Apply Lemma \ref{bc} and rearrange terms.
 \begin{align*}
 &= 2c(F_{\mathcal{VI}(Q) \cup S_{1}}) - v(G) + (e(F_{\mathcal{VI}(Q)}) + \left|S_{1}\right| + \left|S_{2}\right|) - (bc(F_{\mathcal{VI}(Q)})- \left|S_{1}\right| + \left|S_{2}\right|) \\
 &= (2c(F_{\mathcal{VI}(Q)}) - v(G) + e(F_{\mathcal{VI}(Q)}) - bc(F_{\mathcal{VI}(Q)})) \\
 &\; \; \; + (2c(F_{\mathcal{VI}(Q) \cup S_{1}}) + 2\left|S_{1}\right| - 2c(F_{\mathcal{VI}(Q)})) \\
 &= s(F_{\mathcal{VI}(Q)}) + 2n(W)
 \end{align*}
 
 A parallel computation in the dual graph yields
 \begin{equation}
 s(R_{\mathcal{VE}(Q) \cup S^{c}}) = s(R_{\mathcal{VE}(Q)}) + 2n(W^{*})
 \end{equation}
 
 We rewrite the inner sum of \eqref{dblesum}, using these results, as
 \begin{equation}\label{innersum}
 \sum_{S \subseteq \mathcal{I}_{o}(Q) \cup \mathcal{E}_{o}(Q)} X^{c(W)-1} Y^{c(W^{*})-1}
  A^{s(F_{\mathcal{VI}(Q)})/2 + n(W)} B^{s(R_{\mathcal{VE}(Q)})/2 + n(W^{*})}
  \end{equation}
  $W$ depends only on $S_{1}$, and $W^{*}$ depends only on $S_{2}$. Thus we may sum first over $S_{2}$, then over $S_{1}$. \eqref{innersum} becomes
  \begin{align*}
  &\sum_{S_{1} \subseteq \mathcal{I}_{o}(Q)}X^{c(W)-1}A^{s(F_{\mathcal{VI}(Q)})/2 + n(W)} \sum_{S_{2} \subseteq \mathcal{E}_{o}(Q)} Y^{c(W^{*})-1} B^{s(R_{\mathcal{VE}(Q)})/2 + n(W^{*})} \\
  &= A^{s(F_{\mathcal{VI}(Q)})/2} B^{s(R_{\mathcal{VE}(Q)})/2} \sum_{S_{1} \subseteq \mathcal{I}_{o}(Q)}X^{c(W)-1}A^{n(W)} \sum_{S_{2} \subseteq \mathcal{E}_{o}(Q)} Y^{c(W^{*})-1}B^{n(W^{*})} \\
  &= A^{s(F_{\mathcal{VI}(Q)})/2} B^{s(R_{\mathcal{VE}(Q)})/2} T_{G_{Q}}(X,A)T_{G^{*}_{Q^{*}}}(Y,B) 
  \end{align*}
 \end{proof}

\section{Derivation of expansions for other polynomials}
  
 \subsection{Bollob\'as-Riordan Polynomial}
 In this section we will rederive the quasi-tree expansion for the Bollob\'as-Riordan polynomial given in \cite{CKS1}, \cite{VT2}, and \cite{De}, using our expansion for $\pol_{G}$. Recall that, for a ribbon graph $G$, the Bollob\'as-Riordan polynomial is defined by,
 \begin{equation}\label{trans}
 BR_{G}(X,Y,Z) = \sum_{F \subseteq G}X^{c(F)-c(G)}Y^{n(F)}Z^{c(F)+n(F)-bc(F)}
 \end{equation}
 
 where the sum is taken over all spanning ribbon subgraphs $F$ of $G$. Note that the exponent of $Z$ is equal to $s(F)$, which is equal to $2g(F)$ when $F$ is orientable. 
 
The following theorem is a straightforward computation which is carried out in the case of orientable ribbon graphs in \cite{Kru}, but easily extends to the case when $G$ is not orientable by using the formulation \ref{ribbon} of $P_{G}$. 
 
 \begin{thm}
 \[
 BR_{G}(X,Y,Z) = Y^{s(G)/2} \pol_{G}(X,Y,YZ^2,Y^{-1})
 \]
 \end{thm}
 
 We will use this identity to derive the quasi-tree expansion of the Bollob\'as-Riordan polynomial.
 
 \begin{thm}
 
 \[
 BR_{G}(X,Y,Z) = \sum_{Q \in \mathcal{Q}_{G}}Y^{n(F_{\mathcal{VI}(Q)})}Z^{(k-n+bc)(F_{\mathcal{VI}(Q)})}(1+Y)^{\left| \mathcal{E}_{o}(Q) \right|}
 T_{G_{Q}}(X,YZ^2)
 \]
 
 \end{thm}
 
\begin{proof}
We begin with the expansion of $\pol_{G}$ in Theorem \ref{quasi} and apply the substitution of Theorem \ref{trans}. The result of this substitution is,
\begin{equation}\label{bolqua}
 Y^{\dis / 2}\sum_{Q \in \mathcal{Q}_{G}} T_{G_{Q}}(X,YZ^2) T_{G^{*}_{Q^{*}}}(Y,Y^{-1})Z^{s(F_{\mathcal{VI}(Q)})}Y^{\frac{1}{2}(s(F_{\mathcal{VI}(Q)})-s(R_{\mathcal{VE}(Q)}))}
 \end{equation}
 
 Now recall that for any spanning subgraph $F$, $s(F) = 2k(F) - \chi(F)$. Thus $s(F_{\mathcal{VI}(Q)}) = (k-n+bc)(F_{\mathcal{VI}(Q)})$. Next, consider $T_{G^{*}_{Q^{*}}}(Y,Y^{-1})$.
 \begin{align*}
 T_{G^{*}_{Q^{*}}}(Y,Y^{-1}) &= \sum_{W^{*} \subseteq G^{*}_{Q^{*}}}Y^{c(W^{*})-1-n(W^{*})} \\
 &= \sum_{W^{*} \subseteq G^{*}_{Q^{*}}}Y^{v(W^{*})-e(W^{*})-1} \\ 
 &= Y^{c(R_{\mathcal{VE}(Q)})-1}\sum_{W^{*} \subseteq G^{*}_{Q^{*}}}Y^{-e(W^{*})} \\
 &= Y^{c(R_{\mathcal{VE}(Q)})-1}(1+Y^{-1})^{\left| \mathcal{E}_{o}(Q) \right|}
 \end{align*}
 
 Lastly, we consider the quantity $\dis + s(F_{\mathcal{VI}(Q)})-s(R_{\mathcal{VE}(Q)})$. Writing the summands in their combinatorial form, this is equivalent to,
 \begin{align*}
 &2+2(c(F_{\mathcal{VI}(Q)})-c(R_{\mathcal{VE}(Q)}))+\left| E(G) \right| + \left|\mathcal{VI}(Q)\right| - \left|\mathcal{VE}(Q)\right| \\
 &- 2v(G) + bc(G) - bc(G) - bc(F_{\mathcal{VI}(Q)})+bc(R_{\mathcal{VE}(Q)}) \\ 
 &= 2 + 2\left|\mathcal{VI}(Q)\right| - 2v(G)+2c(F_{\mathcal{VI}(Q)}-2c(R_{\mathcal{VE}(Q)})) \\  
 &+ \left|\mathcal{I}_{o}(Q) \right| + \left| \mathcal{E}_{o}(Q) \right| - bc(F_{\mathcal{VI}(Q)})+bc(R_{\mathcal{VE}(Q)}) \\ 
 &= 2 - 2c(R_{\mathcal{VE}(Q)}) + 2n(F_{\mathcal{VI}(Q)}) + \left|\mathcal{I}_{o}(Q) \right| + \left| \mathcal{E}_{o}(Q) \right| - bc(F_{\mathcal{VI}(Q)})+bc(\mathcal{VE}(Q))
 \end{align*}
 
 We have an equality $ bc(R_{\mathcal{VE}(Q)}) = bc(F_{\mathcal{VI}(Q) \cup S} )$ where $S = \mathcal{I}_{o}(Q) \cup \mathcal{E}_{o}(Q)$. From Lemma \ref{bc} we deduce that 
 \begin{equation}
 bc(R_{\mathcal{VE}(Q)}) = bc(F_{\mathcal{VI}(Q)}) - \left|\mathcal{I}_{o}(Q) \right| + \left| \mathcal{E}_{o}(Q) \right|
 \end{equation}
	Thus, we conclude that 
	\begin{equation}
	\dis + s(F_{\mathcal{VI}(Q)})-s(R_{\mathcal{VE}(Q)}) = 2 - 2c(R_{\mathcal{VE}(Q)}) + 2n(F_{\mathcal{VI}(Q)}) + 2\left| \mathcal{E}_{o}(Q) \right|
 \end{equation}
 Summarizing our calculations, equation \ref{bolqua} becomes,
 \begin{equation}
  \sum_{Q \in \mathcal{Q}_{G}} T_{G_{Q}}(X,YZ^2) Y^{c(\mathcal{VE}(Q))-1}(1+Y^{-1})^{\left| \mathcal{E}_{o}(Q) \right|} Z^{s(F_{\mathcal{VI}(Q)})}
  Y^{1 - c(R_{\mathcal{VE}(Q)}) + n(F_{\mathcal{VI}(Q)}) + \left| \mathcal{E}_{o}(Q) \right|}
 \end{equation}
 This is easily seen to be equivalent to the equation of the theorem. 
\end{proof}

 \subsection{Las Vergnas Polynomial}
 The Las Vergnas polynomial is a special case of the Tutte polynomial of a matroid perspective, introduced in \cite{LV}. It was recently shown to be a specialization of the Krushkal polynomial in \cite{ACEMS}. We use this fact and Theorem \ref{quasi} to derive a quasi-tree expansion for the Las Vergnas polynomial. We first give a brief review of matroids in order to define the Las Vergnas polynomial. For more on the theory of matroids, see \cite{Wel}.
 
\begin{defn}\label{mat}
A matroid is a pair $(M,r)$, where $M$ is a finite set and $r: \mathcal{P}(M) \longrightarrow \mathbb{N}\cup \left\{0\right\}$, where $\mathcal{P}$ is the powerset of $M$. The function $r$ must satisfy the following axioms:
\begin{enumerate}
\item
$r(\emptyset) = 0$
\item
For $H \in \mathcal{P}(M)$ and $y \notin H$, we have 
\[r(H \cup \left\{y\right\}) = \left\{ 
\begin{array}{lr} r(H), \;\; \text{or} \\
r(H)+1
\end{array}
\right.
\]
\item
For $y,z \notin H$, if $r(H \cup \left\{y\right\}) = r(H \cup \left\{z\right\}) = r(H)$, then $r(H \cup \left\{y,z\right\}) = r(H)$.
\end{enumerate} 

\end{defn}

Matroids generalize the notion of linear independence in vector spaces and circuits in a graph: We call $H \in \mathcal{P}(M)$ \emph{independent} if $r(H) = \left|H\right|$, and we call $H$ a \emph{circuit} if $r(H) = \left|H\right| - 1$. If we let $M$ be a finite collection of vectors in a vector space and $r$ the function which assigns to any subcollection of vectors the dimension of the subspace spanned by those vectors, then $(M,r)$ is a matroid and the independent sets are precisely the subcollections of linearly independent vectors. 

The example we are more interested in is that of the \emph{cycle matroid} of a graph. For a graph $G$, we let $M = E(G)$ and define $r(F) = v(G)-c(F)$ for $F$ a subset of edges, where $c(F)$ is the number of connected components of the spanning subgraph consisting of only edges in $F$. As may be readily checked, $(E(G), r)$ is a matroid and the circuits of this matroid are the subsets of edges which form a cycle in $G$, while the independent sets are those edge subsets containing no cycles. We will denote this matroid by $\mathcal{C}(G)$.

We may also take the \emph{dual} of a matroid. For a matroid $(M,r)$, we define its dual to be the matroid $(M,r^{*})$, where $r^{*}$ is defined by $r^{*}(H) = \left|H\right| + r(M \backslash H) - r(M)$. The dual of $\mathcal{C}(G)$ is the \emph{bond matroid} of the graph $G$, which we denote $\mathcal{B}(G)$. The circuits of $\mathcal{B}(G)$ are the minimal edge cuts of $G$, called \emph{bonds} of $G$. These are the \emph{minimal} edge subsets which, when removed from $G$, increase the number of connected components of $G$. Lastly, for any rank function $r$, we define the nullity function $n$ by $n(H) = \left|H\right| - r(H)$.

We now restrict our attention to the case of a graph $G$ cellularly embedded in a (not necessarily orientable) surface $\Sigma$ and its dual in $\Sigma$, $G^{*}$. Let $r$ be the rank function of $\mathcal{C}(G)$, and let $\bar{r}$ be the rank function of $\mathcal{B}(G^{*})$. The Las Vergnas polynomial is defined as
\begin{equation}\label{vergnas}
LV_{G, \Sigma}(X,Y,Z) = \sum_{F \subset E(G)}(X-1)^{r(E(G))-r(F)}(Y-1)^{\bar{n}(F)}Z^{(\bar{r}(E(G^{*}))-\bar{r}(F))-(r(E(G))-r(F))}
\end{equation}
There is a natural map $\mathcal{B}(G^{*}) \longrightarrow \mathcal{C(G)}$, given by the identification of corresponding edges of $G$ and $G^{*}$. This map is a special case of a \emph{matroid perspective}, and the Las Vergnas polynomial is the \emph{Tutte polynomial} of this matroid perspective \cite{LV}.  
The following theorem was shown in \cite{ACEMS}, in the particular case that $\Sigma$ is orientable. 
\begin{thm}\label{kruver}
\[
LV_{G,\Sigma}(X,Y,Z) = Z^{g(\Sigma)}P_{G,\Sigma}(X-1,Y-1,Z^{-1},Z)
\]
\end{thm}
We will show that the following holds in the case of the generalized Krushkal polynomial. The proof is formally the same as that of Theorem \ref{kruver}, as the hypothesis of orientability is only imposed by the fact that the Krushkal polynomial was only defined for orientable surfaces at the time.

\begin{thm}\label{genkruver}
For $\Sigma$ a compact surface, $G$ a cellularly embedded graph in $\Sigma$,
 \[
 LV_{G,\Sigma}(X,Y,Z) = Z^{\dis /2}\pol_{G,\Sigma}(X-1,Y-1,Z^{-1},Z^{1})
 \]
 \end{thm}
 
 \begin{proof}
We analyze the monomials of the right hand side. 
\begin{equation}
Z^{\dis /2}\pol_{G,\Sigma}(X-1,Y-1,Z^{-1/2},Z^{1/2}) = \sum_{F \subseteq G}(X-1)^{c(F)-c(G)}(Y-1)^{c(\Sigma \backslash F)-c(\Sigma)}Z^{\dis/2-s(F)/2+s^{\perp}(F)/2}
\end{equation}
We claim that the equality holds on monomials, where we identify $F \subseteq E(G)$ with its corresponding spanning subgraph. For the exponent of $X$, we have
\begin{equation}
r(E(G))-r(F) = v(G)-c(G)-v(G)+c(F) = c(F) - c(G)
\end{equation}

Now consider the exponent of $Y$ on the left. We have $\bar{n}(F) = \left|F\right|-\bar{r}(F)$. $\mathcal{B}(G^{*})$ is the dual of the cycle matroid $\mathcal{C}(G^{*})$ of $G^{*}$, so we compute 
\begin{equation}
\bar{r}(F) = \left|F\right| + r(E(G) \backslash F) - r(E(G)) = \left|F\right| + c(G^{*}) - c(G^{*} \backslash F)
\end{equation}
so that $\bar{n}(F) = c(F^{*}) - c(G^{*}) = c(\Sigma \backslash F) - c(\Sigma)$

Lastly we consider the exponent of $Z$ on the left. 
\begin{align*}
(\bar{r}(E(G^{*}))-\bar{r}(F))-(r(E(G))-r(F)) &= (\left|E(G) \backslash F\right| - v(G^{*}) + c(F^{*})) - (c(F)-c(G)) \\
&= n(F^{*})-(c(F)-c(G))
\end{align*}
Meanwhile, on the right hand side, we apply equation \ref{comb} in the dual graph $G^{*}$ to obtain 
\begin{align*}
\dis/2-s(F)/2+s^{\perp}(F)/2 &=\dis/2-s^{\perp}(F^{*})/2+s(F^{*})/2 \\
&= n(F^{*}) - c(\Sigma \backslash F^{*}) + c(\Sigma)
\end{align*}
It's clear that $c(\Sigma \backslash F^{*}) = c(F)$. 
\end{proof}

Now we can give the promised quasi-tree expansion of the Las Vergnas polynomial.

\begin{thm}
For $G$ a connected graph cellularly embedded in a compact surface $\Sigma$, we have
\[
LV_{G,\Sigma}(X,Y,Z) = \sum_{Q \in \mathcal{Q}_{G}} T_{G_{Q}}(X,Z^{-1}) T_{G^{*}_{Q^{*}}}(Y,Z) Z^{n(R_{\mathcal{VE}(Q)})-n(G_{Q})}
\]
\end{thm}

\begin{proof}
Theorem \ref{quasi} and Theorem \ref{genkruver} combine to give
\begin{equation}
LV_{G,\Sigma}(X,Y,Z) = \sum_{Q \in \mathcal{Q}_{G}} T_{G_{Q}}(X,Z^{-1}) T_{G^{*}_{Q^{*}}}(Y,Z) Z^{\dis/2-s(F_{\mathcal{VI}(Q)})/2 + s(R_{\mathcal{VE}(Q)})/2}
\end{equation}
We see that the theorem will be true provided that 
\begin{equation}
n(R_{\mathcal{VE}(Q)})-n(G_{Q}) = \dis/2-s(F_{\mathcal{VI}(Q)})/2 + s(R_{\mathcal{VE}(Q)})/2
\end{equation}
Referring back to equation (ref), we have that
\begin{equation}
s(F_{\mathcal{VI}(Q) \cup \mathcal{I}_{o}(Q) \cup \mathcal{E}_{o}(Q)}) = s(F_{\mathcal{VI}(Q)}) + 2n(G_{Q})
\end{equation}
Thus, using equation \ref{comb},
\begin{align*}
\dis/2-s(F_{\mathcal{VI}(Q)})/2 + s(R_{\mathcal{VE}(Q)})/2 &= (\dis/2 - s(F_{\mathcal{VI}(Q) \cup \mathcal{I}_{o}(Q) \cup \mathcal{E}_{o}(Q)})/2 + s(R_{\mathcal{VE}(Q)})/2) -n(G_{Q}) \\
&= (\dis/2 - s^{\perp}(R_{\mathcal{VE}(Q)})/2 + s(R_{\mathcal{VE}(Q)})/2) -n(G_{Q}) \\
&= n(R_{\mathcal{VE}(Q)})-c((R_{\mathcal{VE}(Q)})^{*}) + 1 - n(G_{Q})
\end{align*}

We have $(R_{\mathcal{VE}(Q)})^{*} = F_{\mathcal{VI}(Q) \cup \mathcal{I}_{o}(Q) \cup \mathcal{E}_{o}(Q)}$. Since $F_{\mathcal{VI}(Q) \cup \mathcal{I}_{o}(Q) \cup \mathcal{E}_{o}(Q)}$ contains the quasi-tree $F_{\mathcal{VI}(Q) \cup \mathcal{I}_{o}(Q)}$ as a subgraph, it is connected, so $c((R_{\mathcal{VE}(Q)})^{*}) = 1$.
\end{proof}


\bigskip
\begin{thebibliography}{ABCD}

\bibitem[ACEMS]{ACEMS} R. Askanazi, S. Chmutov, C. Estill, J. Michel, P. Stollenwerk, {\it Polynomial Invariants of Graphs on Surfaces}, To appear in Quantum Topology. Preprint \verb#arXiv:1012.5053v1 [math.CO]#.

\bibitem[BR]{BR} B.~Bollob\'as and O.~Riordan, {\it A polynomial of graphs
   on surfaces}, Math.~Ann.~{\bf 323} (2002) 81--96.

\bibitem[BBC]{BBC} R. Bradford, C. Butler, S. Chmutov, {\it Arrow Ribbon Graphs}, To appear in Journal of Knot Theory and Ramifications.
           Preprint \verb#arXiv:1107.3237v1 [math.CO]#

\bibitem[Ch]{Ch} S.~Chmutov, {\it Generalized duality for graphs on 
   surfaces and the signed Bollob\'as-Riordan polynomial}, 
   Journal of Combinatorial Theory, Ser. B, {\bf 99}(3) (2009) 617--638.
    Preprint \verb#arXiv:math.CO/0711.3490#, 

\bibitem[CKS1]{CKS1} A. Champanerkar, I. Kofman, N. Stoltzfus, {\it Quasi-tree expansion for the Bollob\'as-Riordan-Tutte polynomial}, Bulletin of the London Mathematical Society {\bf 43}(5) (2011) 972--984.

\bibitem[CKS2]{CKS2} A. Champanerkar, I. Kofman, N. Stoltzfus, {\it Graphs on surfaces and Khovanov homology}, Algebraic and Geometric Topology {\bf 7} (2007), 153--1540.

\bibitem[De]{De} E. Dewey, {\it A quasitree expansion of the Bollob\'as-Riordan polynomial}, unpublished work, available at http://www.math.wisc.edu/dewey/

\bibitem[EMM]{EMM} J.~Ellis-Monaghan, I.~Moffatt, {\it Twisted duality 
   and polynomials of embedded graphs}.
   Preprint \verb#arXiv:math.CO/0906.5557#.

\bibitem[Kru]{Kru} V. Krushkal, {\it Graphs, Links, and Duality on Surfaces}, Combinatorics, Probability and Computing {\bf 20} (2011) 267--287.
                 Preprint \verb#arXiv:0903.5312v3 [math.CO]#.

\bibitem[KrRe]{KrRe} V. Krushkal, D. Renardy, {\it A polynomial invariant and duality for triangulations}, Preprint \verb#arXiv:math.CO/1012.1310v2 [math.CO]#.

\bibitem[LV]{LV} M. Las Vergnas, {\it On the Tutte polynomial of a morphism of matroids}, Annals of Discrete Mathematics {\bf 8} (1980) 7-20   
   
\bibitem[Mo]{Mo2} I.~Moffatt, {\it Partial duality and Bollob\'as and 
   Riordan's ribbon graph polynomial},
   Discrete Mathematics {\bf 310} (2010) 174--183.
   Preprint \verb#arXiv:math.CO/0809.3014#.
         
\bibitem[Tu]{Tu} W. Tutte, {\it A contribution to the theory of chromatic polynomials}, Canad. J. Math. {\bf 6} (1953) 80-91.
   
\bibitem[VT]{VT2} F.~Vignes-Tourneret, {\it Non-orientable quasi-trees for the Bollob\'as-Riordan polynomial},  European Journal of Combinatorics {\bf 32}(4) (2011) 510--532.
   
\bibitem[Wel]{Wel} D. J. A. Welsh, {\it Matroid Theory}, Academic Press, London, New York, 1976.

\end{thebibliography}
\end{document}